\documentclass[12pt]{amsart}
\usepackage[margin=1.2in]{geometry}
\usepackage{graphicx,latexsym}
\usepackage{comment}
\usepackage{mathrsfs}
\usepackage{tikz}
\usepackage{amsfonts, amssymb, amsmath, amsthm, bm, hyperref}
\usepackage[utf8]{inputenc}
\usepackage[
    backend=biber,
    style=numeric,
    natbib=false,
    giveninits=true,
    url=false, 
    doi=true,
    eprint=true
]{biblatex}

\renewbibmacro{in:}{%
  \ifentrytype{article}{}{\printtext{\bibstring{in}\intitlepunct}}}

\newcommand{\N}{\mathbb{N}}

\newcommand{\R}{\mathbb{R}}

\newcommand{\cS}{\mathcal{S}}
\newcommand{\cO}{\mathcal{O}}

\newcommand{\dx}{{\rm d}x }

\newcommand{\dt}{{\rm d}t }
\newcommand{\dxi}{{\rm d}\xi }

\newcommand{\e}{\varepsilon}

\newcommand{\Vpol}{\mathcal{V}_{\operatorname{pol}}}
\newcommand{\CCS}{\operatorname{CCS}}
\newcommand{\coleq}{\mathrel{\mathop:}=}

\newcommand{\abso}[1]{\left|#1\right|}
\newcommand{\norm}[1]{\lVert#1\rVert}

\DeclareMathOperator{\id}{id}

\newtheorem{theorem}{Theorem}[section]
\newtheorem{proposition}[theorem]{Proposition}

\newtheorem{lemma}[theorem]{Lemma}

\theoremstyle{definition}

\theoremstyle{remark}
\newtheorem{remark}[theorem]{Remark}

\newtheorem{example} [theorem]{Example}

\numberwithin{equation}{section}

\begin{document}
\title{On the space of Laplace transformable distributions}

\author[A.~Debrouwere]{Andreas Debrouwere}
\address{A.~Debrouwere, Department of Mathematics: Analysis, Logic and Discrete Mathematics\\ Ghent University\\ Krijgslaan 281\\ 9000 Gent\\ Belgium}
\email{andreas.debrouwere@UGent.be}

\author[E.~A.~Nigsch]{Eduard A.\ Nigsch}
\address{E.~A.~Nigsch \\ Faculty of Mathematics \\ University of Vienna \\ Oskar-Morgenstern-Platz 1\\ 1090 Wien \\Austria}
\email{eduard.nigsch@univie.ac.at}

\subjclass[2010]{\emph{Primary.}  46F05, 46A13. \emph{Secondary} 81S30.} 
\keywords{Laplace transform; distributions; ultrabornological (PLS)-spaces; short-time Fourier transform}

\begin{abstract} 
We show that the space $\cS'(\Gamma)$ of Laplace transformable distributions, where $\Gamma \subseteq \R^d$ is a non-empty convex open set, is an ultrabornological (PLS)-space. Moreover, we determine an explicit topological predual of $\cS'(\Gamma)$.
\end{abstract}
\maketitle

\section{Introduction}

L.~Schwartz introduced the space $\cS'(\Gamma)$ of Laplace transformable distributions as
\begin{equation*}
\mathcal{S}'(\Gamma) = \{ f \in \mathcal{D}'(\R^d)\ |\ e^{-\xi \cdot x} f(x) \in \mathcal{S}'(\R^d_x)\ \forall \xi \in \Gamma \},
\end{equation*}
where $\Gamma \subseteq \R^d$ is a non-empty convex set \cite[p.~303]{SCH3}. This space is endowed with the projective limit topology with respect to the mappings $\mathcal{S}'(\Gamma) \to \mathcal{S}'(\R^d)$, $f \mapsto e^{-\xi \cdot x} f(x)$ for $\xi \in \Gamma$. The second author together with M.~Kunzinger and N.~Ortner \cite{laplace} recently presented two new proofs of Schwartz's exchange theorem for the Laplace transform of vector-valued distributions \cite[Prop.\ 4.3, p.~186]{SCH2}. Their methods required them to show that $\mathcal{S}'(\Gamma)$ is complete, nuclear and dual-nuclear \cite[Lemma 5]{laplace}. Following a suggestion of N.~Ortner, in this article, we further study the locally convex structure of the space $\mathcal{S}'(\Gamma)$. 


In order to be able to apply functional analytic tools such as De Wilde's open mapping and closed graph theorems \cite[Theorem 24.30 and Theorem 24.31]{M-V} or the theory of the derived project limit functor \cite{Wengenroth}, it is important to determine when a space is ultrabornological.  This is usually straightforward if the space is given by a suitable inductive limit; in fact, ultrabornological spaces are exactly the inductive limits of Banach spaces \cite[Proposition 24.14]{M-V}. The situation for projective limits, however, is more complicated. 
Particularly, this applies to the class  of (PLS)-spaces (i.e., countable projective limits of (DFS)-spaces). The problem of ultrabornologicity has  been extensively studied in this class, both from an abstract point of view as for concrete function and distribution spaces; see the survey article \cite{Domanski} of Doma\'nski and the references therein. 


In the last part of his doctoral thesis \cite[Chap.\ II, Thm.\ 16, p.\ 131]{G2},  A.~Grothendieck showed that  the convolutor space $\cO_C'$ is ultrabornological. He proved  that $\cO_C'$ is isomorphic to a complemented subspace of the sequence space $s \widehat{\otimes} s'$ and verified directly that the latter space is ultrabornological. Much later, a different proof was given by J.~Larcher and J.~Wengenroth using homological methods \cite{zbMATH06408733}. 
The first author and J.~Vindas \cite{D-V-2018} extended this result  to a considerably wider setting by studying the locally convex structure of  a general class of weighted convolutor spaces. More precisely, they characterized when such spaces are ultrabornological and determined explicit topological preduals for them. One of their main tools is a topological description of these convolutor spaces in terms of the short-time Fourier transform (STFT). 

In this work, we will identify $\mathcal{S}'(\Gamma)$ with a particular instance of the convolutor spaces considered in \cite{D-V-2018}. To this end, we make a detailed study of the mapping properties of the STFT on $\mathcal{S}'(\Gamma)$. Once this identification has been established, we use Theorem 1.1 from \cite{D-V-2018} (see also Theorem \ref{top-char} below) to show that $\cS'(\Gamma)$ is an ultrabornological (PLS)-space and that it admits a weighted (LF)-space of smooth functions on $\R^d$ as a topological predual.


\section{Weighted spaces of continuous functions}

For formulating the mapping properties of the STFT we recall the following notions from \cite{B-M-S} and \cite{D-V-2018}.

Each non-negative function $v$ on $\R^d$ defines a weighted seminorm on $C(\R^d)$ by
\begin{equation*}
 \norm{f}_{v} \coleq \sup_{x \in \R^d} \abso{f(x)} v(x). 
\end{equation*}
We endow the space
\[ Cv(\R^d) \coleq \{ f \in C(\R^d)\ |\ \norm{f}_{v} < \infty \} \]
with this seminorm; it is a Banach space if $v$ is positive and continuous. A pointwise decreasing sequence $\mathcal{V} = (v_N)_{N \in \N}$ of positive continuous functions on $\R^d$ is called a \emph{decreasing weight system}. With this, we define the $(LB)$-space
\[
\mathcal{V}C(\R^d)\coleq \varinjlim_{N \in \N} Cv_N(\R^d).
\]
We consider the following condition on a decreasing weight system $\mathcal{V}$, see \cite[p.\ 114]{B-M-S}:
\begin{equation}\label{condV}\tag{V}
\forall N \in \N  \, \exists M > N \, : \,  \lim_{\abso{x} \to \infty}\frac{v_M(x)}{v_N(x)} = 0.
\end{equation}
The \emph{maximal Nachbin family associated with $\mathcal{V}$} is defined to be the family $\overline{V}=\overline{V}(\mathcal{V})$ consisting of all non-negative upper semicontinuous functions $v$ on $\R^d$ such that 
\[
\forall N \in \N \, : \,\sup_{x \in \R^d} \frac{v(x)}{v_N(x)} < \infty.
\]
The \emph{projective hull of $\mathcal{V}C(\R^d)$} is defined as
\[ C\overline{V}(\R^d) \coleq \{ f \in C(\R^d)\ |\ \norm{f}_{v} < \infty\ \forall v \in \overline{V} \}. \]
and endowed with the locally convex topology generated by the system of seminorms $\{ \norm{\,\cdot \,}_{v} \, | \, v \in \overline{V} \}$. The spaces $\mathcal{V}C(\R^d)$ and $C\overline{V}(\R^d)$ always coincide as sets and, if  $\mathcal{V}$ satisfies condition \eqref{condV}, also as locally convex spaces \cite[Thm.\ 1.3 (d), p.~118]{B-M-S}.

A pointwise increasing sequence $\mathcal{W} = (w_N)_{N \in \N}$ of  positive continuous functions on $\R^d$ is called an \emph{increasing weight system}. Given such a system, we define the Fr\'echet space
\[
\mathcal{W}C(\R^d)\coleq \varprojlim_{N \in \N} Cw_N(\R^d).
\]
We consider the following conditions on an increasing weight system $\mathcal{W}$:
\begin{gather}
\forall N \in \N  \, \exists M > N \, : \,  \lim_{\abso{x} \to \infty}\frac{w_N(x)}{w_M(x)} = 0, \label{V} \\
\forall N \in \N  \, \exists M > N : \frac{w_N}{w_M} \in L^1(\R^d), \label{L1-cond} \\
\forall N  \in \N  \, \exists M_1,M_2 \geq N \, \exists C > 0 \, \forall x,y \in \R^d : w_N(x+y) \leq C w_{M_1}(x) w_{M_2}(y). \label{trans-inv}
\end{gather}

In the next lemma,  we obtain a concrete representation of the $\e$-tensor product of weighted spaces of continuous functions.\begin{lemma}\label{proj-desc}
Let $\mathcal{W} = (w_N)_{N \in \N}$ be an increasing weight system and $\mathcal{V} = (v_n)_{n \in \N}$ a decreasing weight system satisfying $(V)$. Then, we have the identification
 \[ \mathcal{W}C(\R^d_x) \widehat{\otimes}_\varepsilon \mathcal{V}C(\R^d_\xi) = \{ f \in C(\R^{2d}_{x,\xi})\ |\ \forall N  \in \N\ \exists n \in \N: \norm{f}_{w_N, v_n}  < \infty \}, \]
where we set $\norm{f}_{w \otimes v} \coleq \sup_{(x,\xi) \in \R^{2d}} \abso{f(x,\xi)}w(x) v(\xi)$ for non-negative functions $w,v$ on $\R^d$. Moreover, $f \in C(\R^{2d}_{x,\xi})$  belongs to $\mathcal{W}C(\R^d_x) \widehat{\otimes}_\varepsilon \mathcal{V}C(\R^d_\xi)$ if and only if $\norm{f}_{w_N \otimes v} < \infty$ for all $N \in \N$ and $v \in \overline{V}$.
Consequently, the topology of $\mathcal{W}C(\R^d_x) \widehat{\otimes}_\varepsilon \mathcal{V}C(\R^d_\xi)$ is generated by the system of seminorms $\{ \norm{\, \cdot \, }_{w_N \otimes v} \, | \, N \in \N, v \in \overline{V} \}$. 
\end{lemma}
\begin{proof}
This follows from the fact that the $\varepsilon$-tensor product commutes with projective limits and  \cite[Thm.\ 3.1 (c), p.~137]{B-M-S}.
\end{proof}

\section{The short-time Fourier transform on $\mathcal{D}'(\R^d)$}
The translation and modulation operators are denoted by $T_xf(t) = f(t-x)$ and $M_\xi f(t) = e^{2\pi i \xi \cdot t} f(t)$ for $x, \xi \in \R^d$. The \emph{short-time Fourier transform (STFT)} of a function $f \in L^2(\R^d)$ with respect to a window function $\psi \in L^2(\R^d)$ is defined as
\[
V_\psi f(x,\xi) \coleq (f, M_\xi T_x\psi)_{L^2} = \int_{\R^d} f(t) \overline{\psi(t-x)}e^{-2\pi i \xi \cdot t}\, \dt, \qquad (x, \xi) \in \R^{2d},
\]
where $(\cdot,\cdot)_{L^2}$ denotes the inner product on $L^2(\R^d)$. We have that $\norm{V_\psi f}_{L^2(\R^{2d})} = \norm{\psi}_{L^2}\norm{f}_{L^2}$. In particular, the mapping $V_\psi \colon L^2(\R^d) \rightarrow L^2(\R^{2d})$ is continuous. The adjoint of $V_\psi$ is given by the weak integral
\[
V^\ast_\psi F = \int \int_{\R^{2d}} F(x,\xi) M_\xi T_x\psi\, \dx\, \dxi, \qquad F \in L^2(\R^{2d}).
\]
If $\psi \neq 0$ and $\gamma \in L^2(\R^d)$ is a synthesis window for $\psi$, that is, $(\gamma, \psi)_{L^2} \neq 0$, then
\[
\frac{1}{(\gamma, \psi)_{L^2}} V^\ast_\gamma \circ V_\psi = \id_{L^2(\R^d)}.
\]
We refer to \cite{Grochenig} for further properties of the STFT. 

Next, we explain how the STFT can be extended to the space of distributions; see \cite[Sect.\ 2]{D-V-2018} for details and proofs. We set
$\mathcal{V}_{\operatorname{pol}} = ((1+ \abso{\, \cdot \, })^{-N})_{N \in \N}.$
Fix a window function $\psi \in \mathcal{D}(\R^d)$. For $f \in \mathcal{D}'(\R^d)$ we define
\[
V_\psi f(x,\xi) \coleq  \langle f, \overline{M_\xi T_x\psi} \rangle, \qquad (x,\xi) \in \R^{2d}.
\]
Clearly, $V_\psi f$ is a continuous function on $\R^{2d}$. In fact,
\[
V_\psi \colon \mathcal{D}'(\R^d) \rightarrow C(\R^d_x) \widehat{\otimes}_\varepsilon \Vpol C(\R^d_\xi)
\]
is a well-defined continuous mapping \cite[Lemma 2.2]{D-V-2018}. We \emph{define} the adjoint STFT of an element $F \in C(\R^d_x) \widehat{\otimes}_\varepsilon \Vpol C(\R^d_\xi)$ as the distribution
\[
\langle V^\ast_\psi F, \varphi \rangle \coleq \int \int_{\R^{2d}} F(x,\xi) V_{\overline{\psi}}\varphi(x, -\xi)\,\dx\,\dxi, \qquad \varphi \in \mathcal{D}(\R^d). 
\]
Then,
\[
V^\ast_\psi \colon C(\R^d_x) \widehat{\otimes}_\varepsilon \Vpol C(\R^d_\xi) \rightarrow \mathcal{D}'(\R^d)
\]
is a well-defined continuous mapping by \cite[Prop.\ 2.2]{D-V-2018}. Finally, if $\psi \neq 0$ and $\gamma \in \mathcal{D}(\R^d)$  is a synthesis window for $\psi$, then the following reconstruction formula holds \cite[Prop.\ 2.4]{D-V-2018}:
\begin{equation}
\frac{1}{(\gamma, \psi)_{L^2}} V^\ast_\gamma \circ V_\psi = \operatorname{id}_{\mathcal{D}'(\R^d)}.
\label{reconstruction-D-dual}
\end{equation}
\section{Duals of inductive limits of weighted spaces of smooth functions}\label{prev-paper}
Let $v$ be a non-negative function on $\R^d$ and $n \in \N$. We define $\mathcal{B}^n_v(\R^d)$ as the seminormed space consisting of all $\varphi \in C^n(\R^d)$ such that
\[
\norm{\varphi}_{v,n} \coleq \max_{\abso{\alpha} \leq n} \sup_{x \in \R^d} \abso{\partial^{\alpha}\varphi(x)}v(x) < \infty. 
\]
As before, $\mathcal{B}^n_v(\R^d)$ is a Banach space if  $v$ is positive and continuous. Let $\mathcal{W} = (w_N)_{N \in \N}$ be an increasing weight system. We define the (LF)-space
\[
\mathcal{B}_{\mathcal{W}^\circ}(\R^d) := \varinjlim_{N \in \N} \varprojlim_{n \in \N} \mathcal{B}^n_{1/w_N}(\R^d).
\]
We endow the dual space $\mathcal{B}'_{\mathcal{W}}(\R^d) \coleq (\mathcal{B}_{\mathcal{W}^\circ}(\R^d))'$ with the strong topology. If $\mathcal{W}$ satisfies \eqref{V}, then $\mathcal{D}(\R)$ is densely and continuously included in $\mathcal{B}_{\mathcal{W}^\circ}(\R^d)$ and therefore  $\mathcal{B}'_{\mathcal{W}}(\R^d)$ is a vector subspace of $\mathcal{D}'(\R^d)$. 

On the other hand, we define the convolutor space
\[ \mathcal{O}'_{C,\mathcal{W}}(\R^d) \coleq \{ f \in \mathcal{D}'(\R^d)\ |\ f \ast \varphi \in\mathcal{W}C(\R^d)\ \forall \varphi \in \mathcal{D}(\R^d) \}. \]
For $f \in \mathcal{O}'_{C,\mathcal{W}}(\R^d)$ fixed, the mapping 
\[
\mathcal{D}(\R^d) \rightarrow {\mathcal{W}}C(\R^d),\quad \varphi \mapsto f \ast \varphi
\]
 is continuous, as follows from the closed graph theorem. We endow $\mathcal{O}'_{C,\mathcal{W}}(\R^d)$ with the topology induced via the embedding
\[
\mathcal{O}'_{C,\mathcal{W}}(\R^d) \rightarrow L_\beta(\mathcal{D}(\R^d), \mathcal{W}C(\R^d)),\quad  f \mapsto [\varphi \mapsto f \ast \varphi],
\]
where $\beta$ denotes the topology of uniform convergence on bounded sets.

In \cite{D-V-2018} the structural and topological properties of the spaces $\mathcal{B}'_{\mathcal{W}}(\R^d)$ and $\mathcal{O}'_{C,\mathcal{W}}(\R^d)$ are discussed. We now present the main results of this paper and refer to \cite{D-V-2018} for more details and proofs\footnote{To be precise, the spaces considered in \cite{D-V-2018}, denoted there by $(\dot{\mathcal{B}}_{\mathcal{W}^\circ} (\R^d))'$ and $\mathcal{O}'_C(\mathcal{D}, L^1_{\mathcal{W}})$, differ from  $\mathcal{B}'_{\mathcal{W}}(\R^d)$ and $
\mathcal{O}'_{C,\mathcal{W}}(\R^d)$ defined above. However, if  $\mathcal{W}$ satisfies \eqref{V}, \eqref{L1-cond} and \eqref{trans-inv}, then $\mathcal{B}'_{\mathcal{W}}(\R^d) = (\dot{\mathcal{B}}_{\mathcal{W}^\circ} (\R^d))'$ and $\mathcal{O}'_C(\mathcal{D}, 
L^1_{\mathcal{W}}) = \mathcal{O}'_{C,\mathcal{W}}(\R^d)$; the first equality is clear, while the second one follows from \cite[Prop.\ 6.2]{D-V-2018}. Moreover, under these conditions, all statements and proofs from \cite{D-V-2018} remain valid if one replaces  $L^1_{\mathcal{W}}(\R^d)$ 
by  $\mathcal{W}C(\R^d)$.}.

\begin{proposition}  \cite[Prop.\ 4.2]{D-V-2018}\label{STFT-char-1}
Let $\mathcal{W}$ be an increasing weight system satisfying \eqref{V}, \eqref{L1-cond} and \eqref{trans-inv} and let $\psi \in \mathcal{D}(\R^d)$.  Then, the mappings
\[
V_\psi\colon \mathcal{O}'_{C,\mathcal{W}}(\R^d) \rightarrow \mathcal{W}C(\R^d_x) \widehat{\otimes}_\varepsilon \Vpol C(\R^d_\xi)
\]
and
\[
V^*_\psi\colon  \mathcal{W}C(\R^d_x) \widehat{\otimes}_\varepsilon \Vpol C(\R^d_\xi) \rightarrow  \mathcal{O}'_{C,\mathcal{W}}(\R^d)
\]
are well-defined and continuous.
\end{proposition}

\begin{theorem}  \cite[Thm.\ 3.4, Thm.\ 4.6 and Thm.\ 4.15]{D-V-2018} \label{top-char}
Let $\mathcal{W} = (w_N)_{N \in \N}$ be an increasing weight system satisfying \eqref{V}, \eqref{L1-cond} and \eqref{trans-inv}. Then, $\mathcal{B}'_{\mathcal{W}}(\R^d) = \mathcal{O}'_{C,\mathcal{W}}(\R^d)$ as sets and the inclusion mapping $\mathcal{B}'_{\mathcal{W}}(\R^d) \rightarrow \mathcal{O}'_{C,\mathcal{W}}(\R^d)$ is continuous. Moreover, the following statements are equivalent:
\begin{itemize}
\item[$(i)$] $\mathcal{B}'_{\mathcal{W}}(\R^d) = \mathcal{O}'_{C,\mathcal{W}}(\R^d)$ as locally convex spaces.
\item[$(ii)$] $\mathcal{O}'_{C,\mathcal{W}}(\R^d)$ is an ultrabornological (PLS)-space.
\item[$(iii)$] The (LF)-space $\mathcal{B}_{\mathcal{W}^\circ}(\R^d)$ is complete.
\item[$(iv)$] $\mathcal{W}$ satisfies
\begin{gather}
 \label{Omega-switched}
\forall N \in \N  \, \exists M \geq N \, \forall P \geq M \, \exists \theta \in (0,1) \, \exists C > 0 \, \forall x \in \R^d:  \\ 
{w_N(x)}^{1-\theta}{w_P(x)}^{\theta} \leq Cw_M(x). \nonumber
\end{gather}
\end{itemize}
\end{theorem}
\begin{remark} Condition \eqref{Omega-switched} is closely connected with Vogt's condition $(\Omega)$ that plays an essential role in the structure and splitting theory for Fr\'echet spaces.
\end{remark}

\section{The space $\mathcal{S}'(\Gamma)$}

Our next goal is to characterize $\mathcal{S}'(\Gamma)$ in terms of the STFT.  

Let $\emptyset \neq \Gamma \subseteq \R^d$ be open and convex. We denote by $\CCS(\Gamma)$ the family of all non-empty compact convex subsets of $\Gamma$ and by $\mathfrak{B}(\mathcal{S}(\R^d))$ the family of all bounded subsets of $\mathcal{S}(\R^d$). The topology of $\mathcal{S}'(\Gamma)$ can also be described as follows.

\begin{lemma} \cite[p.\ 301]{SCH3} Let $\emptyset \neq \Gamma \subseteq \R^d$ be open and convex.  For all $K \in \CCS(\Gamma)$ and $B \in \mathfrak{B}(\mathcal{S}(\R^d))$ we have that
\[
p_{K,B}(f) \coleq \sup_{\eta \in K} \sup_{\varphi \in B} \abso{\langle e^{-\eta \cdot x}f(x), \varphi(x) \rangle} < \infty, \qquad f \in \mathcal{S}'(\Gamma).
\]
Moreover, the topology of $\mathcal{S}'(\Gamma)$ is generated by the system of seminorms $\{p_{K,B} \, | \,  K \in \CCS(\Gamma), B \in \mathfrak{B}(\mathcal{S}(\R^d))\}$.
\end{lemma}
We need to introduce some additional terminology. Given a  non-empty compact convex subset $K$ of $\R^d$, we define its \emph{supporting function} as
\[
h_K(x) =\max_{\eta \in K} x \cdot \eta, \qquad x \in \R^d.
\]
It is clear from the definition that $h_K$ is subadditive and positive homogeneous of degree one. In particular, $h_K$ is convex. Supporting functions have  the following elementary properties.
\begin{lemma} {\cite[Cor.\ 1.8.2 and Prop.\ 1.8.3]{Morimoto}} \label{supporting-function}
Let $K_1$ and $K_2$ be non-empty compact convex subsets of $\R^d$. 
\begin{itemize}
\item[$(a)$] $K_1 \subseteq K_2$ if and only if $h_{K_1}(x) \leq h_{K_2}(x)$ for all $x \in \R^d$.
\item[$(b)$] $h_{K_1+K_2}(x) = h_{K_1}(x) + h_{K_2}(x)$ for all $x \in \R^d$.
\end{itemize}
\end{lemma} 
\begin{example}\label{example} For $r > 0$ we have $h_{\overline{B}(0,r)}(x) = r \abso{x}$ for all $x \in \R^d$, where $\overline{B}(0,r)$ denotes the closed ball in $\R^d$ centered at the origin with radius $r$. Next, let $K$ be a non-empty compact convex subset of $\R^d$ and  $\varepsilon > 0$. We set $K_\varepsilon = K + \overline{B}(0,\varepsilon)$. Lemma \ref{supporting-function} and the above  yield that $h_{K_\varepsilon}(x) = h_K(x) + \varepsilon\abso{x}$ for all $x \in \R^d$.
\end{example}
Let $\emptyset \neq \Gamma \subseteq \R^d$ be open and convex and let $(K_N)_{N \in \N} \subset \CCS(\Gamma)$ be such that $K_N \subseteq K_{N+1}$ for all $N \in \N$ and $\Gamma = \bigcup_{N} K_N$. Lemma \ref{supporting-function} yields that $\mathcal{W} = (e^{h_{-K_N}})_{N \in \N}$ is an increasing weight system. We set $C_\Gamma(\R^d) \coleq \mathcal{W}C(\R^d)$. Clearly, the definition of $C_\Gamma(\R^d)$ is independent of the chosen sequence $(K_N)_{N \in \N}$. The next result is the key observation of this article.
\begin{proposition}\label{STFT-char}
Let $\emptyset \neq \Gamma \subseteq \R^d$ be open and convex and let $\psi \in \mathcal{D}(\R^d)$. Then, the mappings
\[
V_\psi\colon \mathcal{S}'(\Gamma) \rightarrow C_\Gamma(\R^d_x) \widehat{\otimes}_\varepsilon \Vpol C(\R^d_\xi)
\]
and
\[
V^*_\psi\colon  C_\Gamma(\R^d_x) \widehat{\otimes}_\varepsilon \Vpol C(\R^d_\xi) \rightarrow  \mathcal{S}'(\Gamma)
\]
are well-defined and continuous.
\end{proposition}
We need some preparation for the proof of Proposition \ref{STFT-char}. Firstly, Lemma \ref{proj-desc} implies that the the topology of $C_\Gamma(\R^d_x) \widehat{\otimes}_\varepsilon \Vpol C(\R^d_\xi)$ is generated by the system of seminorms
\[
\norm{ f }_{K,v} \coleq \sup_{(x,\xi) \in \R^{2d}}   \abso{f(x,\xi)}e^{h_{-K}(x)} v(\xi) < \infty, \qquad K \in CCS(\Gamma), v \in \overline{V}(\Vpol ).
\]
For $k,n \in \N$ we write 
\[
\norm{ \varphi }_{\mathcal{S}^n_k} \coleq \max_{\abso{\alpha} \leq n} \sup_{x \in \R^d} \abso{\partial^\alpha \varphi(x)} (1+\abso{x})^k, \qquad \varphi \in \mathcal{S}(\R^d).
\]
The topology of $ \mathcal{S}(\R^d)$ is generated by the system of seminorms $\{ \norm{ \,\cdot\, }_{\mathcal{S}^n_k} \, | \, k,n \in \N \}$. We now give two technical lemmas.
\begin{lemma}\label{lemma-1}
Let $\psi \in \mathcal{D}(\R^d)$, $K \subset \R^d$ be compact,  $v \in \overline{V}(\Vpol )$ and $\varepsilon > 0$. Then, 
\[
\{ e^{\eta\cdot(t-x)} \overline{M_\xi T_x\psi}(t) e^{-\varepsilon\abso{x}}v(\xi) \, | \, (x,\xi) \in \R^{2d}, \eta \in K\} \in  \mathfrak{B}(\mathcal{S}(\R^d_t)).
\]
\end{lemma}
\begin{proof}
Choose $r > 0$ such that $\operatorname{supp} \psi \subseteq \overline{B}(0,r)$ and $R \geq 1$ such that $K \subseteq \overline{B}(0,R)$. For all $k,n \in \N$ we have that
\begin{align*}
&\sup_{(x,\xi) \in \R^{2d}} \sup_{\eta \in K}e^{-\varepsilon\abso{x}}v(\xi) \norm{ e^{\eta\cdot(t-x)} \overline{M_\xi T_x\psi}(t) }_{\mathcal{S}^n_{k,t}} \leq \sup_{(x,\xi) \in \R^{2d}} \sup_{\eta \in K}e^{-\varepsilon\abso{x}}v(\xi) \cdot \\
&\max_{\abso{\alpha} \leq n} \sup_{x \in \R^d} \sum_{\beta \leq \alpha} \sum_{\gamma \leq \beta} \binom{\alpha}{\beta}  \binom{\beta}{\gamma} \abso{\eta}^{\abso{\alpha} -\abso{\beta}} e^{\eta\cdot(t-x)} (2\pi \abso{\xi})^{\abso{\gamma}} \abso{\partial^{\beta-\gamma} \overline{\psi}(t-x)} (1+\abso{t})^k \\
&\leq  e^{Rr} (8\pi R)^n \max_{\abso{\alpha} \leq n } \norm{ \partial^\alpha \overline{\psi} }_{L^\infty} (1+r)^k \sup_{x \in \R^d} e^{-\varepsilon\abso{x}}(1+\abso{x})^k \sup_{\xi \in \R^d}v(\xi) (1+\abso{\xi})^n \\
&< \infty. \qedhere
\end{align*}
\end{proof}

\begin{lemma}\label{lemma-2}
Let $\psi \in \mathcal{D}(\R^d)$ and $\eta \in \R^d$. Then, for all  $k,n \in \N$ and $\varphi \in \mathcal{S}(\R^d)$,
\[
\abso{V_{\overline{\psi},t}( e^{-\eta \cdot t} \varphi(t))(x, -\xi)} \leq \frac{C_{\eta,k,n,\psi}e^{-\eta\cdot x} \norm{ \varphi }_{\mathcal{S}^n_k}}{(1+\abso{x})^k(1+\abso{\xi})^n}, \qquad (x,\xi) \in \R^{2d},
\]
where
\[
C_{\eta,k,n,\psi} = 4^n(1+\sqrt{d})^n \max\{1,\abso{\eta}^n\} \max_{\abso{\alpha} \leq n} \norm{ \partial^\alpha \psi }_{L^\infty} \int_{\operatorname{supp} \psi } e^{-\eta \cdot t} (1+\abso{t})^k \dt. 
\]
In particular, $\sup_{\eta \in K} C_{\eta,k,n,\psi} < \infty$ for all $K \subset \R^d$ compact.
\end{lemma}
\begin{proof}
We have that
\begin{align*}
&\abso{V_{\overline{\psi},t}( e^{-\eta \cdot t} \varphi(t))(x, -\xi)} (1+\abso{x})^k (1 + \abso{\xi})^n \\
& \leq (1+ \sqrt{d})^n \max_{\abso{\alpha} \leq n}\abso{\xi^\alpha V_{\overline{\psi},t}( e^{-\eta \cdot t} \varphi(t))(x, -\xi)} (1+\abso{x})^k \\
& \leq (1+ \sqrt{d})^n (1+\abso{x})^k\max_{\abso{\alpha} \leq n} \sum_{\beta \leq \alpha} \sum_{\gamma \leq \beta} \binom{\alpha}{\beta}  \binom{\beta}{\gamma} \cdot \\
& \int_{\R^d} \abso{\eta}^{\abso{\gamma}}e^{-\eta \cdot t} \abso{\partial^{\beta-\gamma}\varphi(t)} \abso{\partial^{\alpha-\beta}\psi(t-x)} \dt\\
& \leq (1+ \sqrt{d})^n (1+\abso{x})^k\max_{\abso{\alpha} \leq n} \sum_{\beta \leq \alpha} \sum_{\gamma \leq \beta} \binom{\alpha}{\beta}  \binom{\beta}{\gamma} \cdot \\ 
&\int_{\operatorname{supp} \psi} \abso{\eta}^{\abso{\gamma}}e^{-\eta \cdot (t+x)} \abso{\partial^{\beta-\gamma}\varphi(t+x)} \abso{\partial^{\alpha-\beta}\psi(t)} \dt \\
& \leq C_{\eta,k,n,\psi}e^{-\eta\cdot x} \norm{ \varphi }_{\mathcal{S}^n_k}. \qedhere
\end{align*}

\end{proof}
\begin{proof}[Proof of Proposition \ref{STFT-char}]
$(i)$ \emph{$V_\psi\colon \mathcal{S}'(\Gamma) \rightarrow C_\Gamma(\R^d_x) \widehat{\otimes}_\varepsilon \Vpol C(\R^d_\xi)$ is well-defined and continuous}: Let $K \in \CCS(\Gamma)$ and $v \in \overline{V}(\Vpol)$ be arbitrary. Choose $\varepsilon > 0$ so small that $K_\varepsilon \in \CCS(\Gamma)$ and pick, for $x \in \R^d$ fixed, $\eta_x \in K$ such that $h_{-K}(x) \leq (-\eta_x \cdot x) + 1$. Example \ref{example} implies that, for all $f \in \mathcal{S}'(\Gamma)$ and $(x,\xi) \in \R^{2d}$, 
\begin{align*}
\abso{V_\psi f(x,\xi)} e^{h_{-K}(x)}v(\xi) &= \abso{\langle e^{-(\eta_x - \varepsilon \frac{x}{\abso{x}}) \cdot t} f(t), e^{(\eta_x - \varepsilon \frac{x}{\abso{x}}) \cdot t} \overline{M_\xi T_x\psi}(t) \rangle} e^{h_{-K}(x)}v(\xi)  \\
&\leq e \abso{\langle e^{-(\eta_x - \varepsilon \frac{x}{\abso{x}}) \cdot t} f(t), e^{(\eta_x - \varepsilon \frac{x}{\abso{x}}) \cdot (t-x)} \overline{M_\xi T_x\psi}(t) \rangle} e^{-\varepsilon\abso{x}} v(\xi)  \\
&\leq ep_{K_\varepsilon,B}(f),
\end{align*}
where 
\[
B = \{ e^{\tau\cdot(t-x)} \overline{M_\xi T_x\psi}(t) e^{-\varepsilon\abso{x}}v(\xi) \, | \, (x,\xi) \in \R^{2d}, \tau \in K_\varepsilon\} \in  \mathfrak{B}(\mathcal{S}(\R^d_t))
\]
by Lemma \ref{lemma-1}.

\noindent $(ii)$ \emph{$V^\ast_\psi\colon C_\Gamma(\R^d_x) \widehat{\otimes}_\varepsilon \Vpol C(\R^d_\xi) \rightarrow \mathcal{S}'(\Gamma)$ is well-defined and continuous}: We start by showing that $V^\ast_\psi F \in \mathcal{S}'(\Gamma)$ for all $F \in C_\Gamma(\R_x^d) \widehat{\otimes}_\varepsilon \Vpol C(\R_\xi^d)$. Lemma \ref{lemma-2} implies that, for all $\eta \in \Gamma$,
\[
\langle f_\eta, \varphi \rangle =  \int \int_{\R^{2d}} F(x,\xi) V_{\overline{\psi},t}( e^{-\eta \cdot t} \varphi(t))(x, -\xi) \,\dx\,\dxi, \qquad \varphi \in \mathcal{S}(\R^d),
\]
is a well-defined continous linear functional on $\mathcal{S}(\R^d)$. Since $e^{-\eta \cdot t}V^\ast_\psi F(t) = {f_\eta(t)}|_{\mathcal{D}(\R^d)}$, we obtain that $e^{-\eta \cdot t}V^\ast_\psi F(t) \in \mathcal{S}'(\R^d)$ and that
\[
\langle e^{-\eta \cdot t}V^\ast_\psi F(t), \varphi(t) \rangle =  \int \int_{\R^{2d}} F(x,\xi) V_{\overline{\psi},t}( e^{-\eta \cdot t} \varphi(t))(x, -\xi) \,\dx\,\dxi, \qquad \varphi \in \mathcal{S}(\R^d).
\]
Next, we show that $V^\ast_\psi$ is continuous. Let $K \in \CCS(\Gamma)$ and $B \in \mathfrak{B}(\mathcal{S}(\R^d))$ be arbitrary. Choose $\varepsilon > 0$ so small that $K_\varepsilon \in \CCS(\Gamma)$. Lemma \ref{lemma-2} implies that 
there is $v \in \overline{V}(\Vpol)$ such that 
\[
\abso{V_{\overline{\psi}}(e^{-\eta \cdot t} \varphi(t))(x, -\xi)} \leq e^{h_{-K}(x)} v(\xi), \qquad (x,\xi) \in \R^{2d},
\]
 for all $\eta \in K$ and $\varphi \in B$. Set $w(\xi) = v(\xi) (1+ \abso{\xi})^{d+1} \in \overline{V}(\Vpol)$. Example \ref{example} implies that, for all $F \in C_\Gamma(\R_x^d) \widehat{\otimes}_\varepsilon \Vpol C(\R_\xi^d)$,
\begin{align*}
p_{K,B}(V^\ast_\psi F) &\leq \sup_{\eta \in K} \sup_{\varphi \in B} \int \int_{\R^{2d}}  \abso{F(x,\xi)} \abso{V_{\overline{\psi},t}( e^{-\eta \cdot t} \varphi(t))(x, -\xi)} \,\dx\,\dxi\\
&\leq \int \int_{\R^{2d}} \abso{F(x,\xi)} e^{h_{-K}(x)} v(\xi)  \,\dx\,\dxi \leq C \norm{ F }_{K_\varepsilon, w},
\end{align*}
where
\[
C = \int_{\R^d} e^{-\varepsilon\abso{x}} dx \int_{\R^d} \frac{1}{(1+\abso{\xi})^{d+1}} \dxi. \qedhere
\]
\end{proof}
We now combine Theorem \ref{STFT-char-1} with the results from Section \ref{prev-paper} to study the space $\mathcal{S}'(\Gamma)$. Let $\emptyset \neq \Gamma \subseteq \R^d$ be open and convex and let $(K_N)_{N \in \N} \subset \CCS(\Gamma)$ be such that $K_N \subseteq K_{N+1}$ for all $N \in \N$ and $\Gamma = \bigcup_{N} K_N$. For $\mathcal{W} = (e^{h_{-K_N}})_{N \in \N}$ we set $\mathcal{B}'_\Gamma(\R^d) \coleq \mathcal{B}'_{\mathcal{W}}(\R^d)$ and $\mathcal{O}'_{C, \Gamma}(\R^d) = \mathcal{O}'_{C, \mathcal{W}}(\R^d)$. Clearly, these definitions are independent of the chosen sequence $(K_N)_{N \in \N}$. We are ready to state and prove our main theorem.
\begin{theorem}
Let $\emptyset \neq \Gamma \subseteq \R^d$ be open and convex. Then, $\mathcal{S}'(\Gamma) = \mathcal{B}'_\Gamma(\R^d) = \mathcal{O}'_{C, \Gamma}(\R^d)$ as locally convex spaces and $\mathcal{S}'(\Gamma)$ is an ultrabornological (PLS)-space.
\end{theorem}
\begin{proof}
Let $(K_N)_{N \in \N} \subset \CCS(\Gamma)$ be such that $K_N \subseteq K_{N+1}$ for all $N \in \N$ and $\Gamma = \bigcup_{N} K_N$. Set $\mathcal{W} = (e^{h_{-K_N}})_{N \in \N}$. Lemma \ref{supporting-function} and Example \ref{example} imply that $\mathcal{W}$ satisfies  \eqref{V}, \eqref{L1-cond} and \eqref{trans-inv}. Hence, in view of the reconstruction formula \eqref{reconstruction-D-dual}, the topological identity $\mathcal{S}'(\Gamma) = \mathcal{O}'_{C, \Gamma}(\R^d)$ follows from Proposition \ref{STFT-char-1} and Proposition \ref{STFT-char}. Since $\mathcal{W}$ also satisfies \eqref{Omega-switched} (again by Lemma \ref{supporting-function} and Example \ref{example}), the other statements are a direct consequence of Theorem \ref{top-char}. \end{proof}

{\bf Acknowledgements.} We thank N.~Ortner for suggesting the topic of this paper.
E.~A.~Nigsch acknowledges support by the Austrian Science Fund (FWF) grants P26859 and P30233. A.~Debrouwere was supported by  FWO-Vlaanderen through the postdoctoral grant 12T0519N.

\printbibliography

@article{B-M-S,
	author = "Bierstedt, Klaus Dieter and Meise, Reinhold and Summers, William H.",
	journal = "Trans. Am. Math. Soc.",
	pages = "107--160",
	title = "A projective description of weighted inductive limits",
	volume = "272",
	year = "1982",
	zbl = "0599.46026"
}

@online{D-V-2018,
	author = "Debrouwere, A. and Vindas, J.",
	year = 2019,
	eprint = "1801.09246",
	eprinttype = "arxiv",
	note="Submitted",
	title = "Topological properties of convolutor spaces via the short-time Fourier transform"
}

@article{Domanski,
	author = "Pawel Doma\'nski",
	journal = "Banach Center Publications",
	pages = "51--70",
	title = "Classical PLS-spaces: spaces of distributions, real analytic functions and their relatives",
	volume = "64",
	year = "2004",
%	zbl = "0599.46026"
}

@book{Grochenig,
	author = "Gröchenig, Karlheinz",
	journal = "Appl. Numer. Harmon. Anal.",
	publisher = "Boston, MA: Birkhäuser",
	title = "Foundations of time-frequency analysis",
	year = "2001",
}

@book{M-V,
	author = "Meise, Reinhold and Vogt, Dietmar",
	address = "New York",
%	note = "Translated and revised from the 1976 Japanese original by the author",
	publisher = "Clarendon Press, Oxford University Press",
%	series = "Translations of Mathematical Monographs",
	title = "Introduction to functional analysis",
%	volume = "129",
	year = "1997"
}

@book{Morimoto,
	author = "Morimoto, Mitsuo",
	note = "Translated and revised from the 1976 Japanese original by the author",
	publisher = "American Mathematical Society, Providence, RI",
	series = "Translations of Mathematical Monographs",
	title = "An introduction to Sato's hyperfunctions",
	volume = "129",
	year = "1993"
}

@article{SCH2,
	author = "Schwartz, Laurent",
	doi = "10.5802/aif.77",
	issn = "0373-0956",
	journal = "Ann. Inst. Fourier",
	title = "Th{\'e}orie des distributions {\`a} valeurs vectorielles {II}",
	volume = "8",
	year = "1958"
}

@book{SCH3,
	address = "Paris",
	author = "Schwartz, Laurent",
	edition = "Nouvelle {\'e}dition, enti{\`e}rement corrig{\'e}e, refondue et augment{\'e}e",
	publisher = "Hermann",
	title = "Th{\'e}orie des distributions",
	year = "1966"
}

@article{G2,
	author = "Grothendieck, Alexandre",
	issn = "0065-9266",
	journal = "Mem. Am. Math. Soc.",
	title = "Produits tensoriels topologiques et espaces nucl{\'e}aires. {C}hap. {II}",
	volume = "16",
	year = "1955"
}

@article{zbMATH06408733,
	author = "Larcher, Julian and Wengenroth, Jochen",
	journal = "Bull. Belg. Math. Soc. - Simon Stevin",
	number = "5",
	pages = "887--894",
	title = "A new proof for the bornologicity of the space of slowly increasing functions.",
	volume = "21",
	year = "2014",
}

@article{laplace,
	author = "Kunzinger, Michael and Nigsch, Eduard A. and Ortner, Norbert",
	doi = "10.1016/j.jmaa.2019.06.002",
	journal = "J. Math. Anal. Appl.",
	number = "2",
	pages = "990--1004",
	title = "Laplace transformation of vector-valued distributions and applications to Cauchy-Dirichlet problems",
	volume = "478",
	year = "2019"
}

@book{Wengenroth,
	address = "Berlin",
	author = "Wengenroth, Jochen",
%	isbn = "978-0-486-48850-9",
	publisher = "Springer-Verlag",
	title = "Derived functors in functional analysis",
	year = "2003"
}

\end{document}